\documentclass[12pt]{amsart}

\usepackage{amscd}
\usepackage{amsmath}
\usepackage{amsfonts}
\usepackage{amssymb}
\usepackage{graphicx}
\usepackage{xypic}  % for commutative diagrams

\font\Bbb=msbm10 scaled \magstep 2

\def\C{\hbox{\Bbb C}}
\def\R{\hbox{\Bbb R}}
\def\Z{\hbox{\Bbb Z}}

\def\K{\hbox{\Bbb K}}

\def\Log{\hbox{\rm Log}}

\newtheorem{theorem}{\bf Theorem}

\newtheorem{example}[theorem]{\bf Example}
\newtheorem{proposition}[theorem]{\bf Proposition}

\newtheorem{definition}[theorem]{\bf Definition}
\newtheorem{conjecture}[theorem]{\bf Conjecture}

\newcommand{\val}{\mathop{\rm val}\nolimits}
\newcommand{\Val}{\mathop{\rm Val}\nolimits}

\newcommand{\Arg}{\mathop{\rm Arg}\nolimits}

\newcommand{\Verte}{\mathop{\rm Vert}\nolimits}

\numberwithin{equation}{section} \numberwithin{theorem}{section}

\title{Amoeba-shaped polyhedral complex of an algebraic hypersurface}

\author{Mounir Nisse and Timur Sadykov}

\address{School of Mathematics,
\newline \indent Korea Institue for Advanced Study,
\newline \indent 87 Hoegiro Dongdaemun-gu, Seoul,
\newline \indent 130-722, Republic of Korea}
\email{mounir.nisse@gmail.com}

\address{Department of Mathematics
\newline \indent and Computer Science,
\newline \indent Plekhanov Russian University
\newline \indent 115054, Moscow, Russia}
\email{Sadykov.TM@rea.ru}

\thanks{The second author was supported by the grant of the
President of the Russian Federation for state support of
leading scientific schools NSh-9149.2016.1 and the grant
of the Government of the Russian Federation for investigations
under the guidance of the leading scientists of the
Siberian Federal University (contract No.~14.Y26.31.0006)}

\begin{document}

\begin{abstract}
Given a complex algebraic hypersurface~$H$, we introduce a
polyhedral complex which is a subset of the Newton polytope of the
defining polynomial for~$H$ and enjoys the key topological and
combinatorial properties of the amoeba of~$H.$ We provide an explicit
formula for this polyhedral complex in the case when the spine of the amoeba
is dual to a triangulation of the Newton polytope of the defining polynomial.
In particular, this yields a description of the polyhedral complex when the
hypersurface is optimal~\cite{FPT}.
\end{abstract}

\maketitle

\section*{Introduction}

Amoebas of complex algebraic varieties have attracted
substantial attention in the recent years after their
inception in~\cite{GKZ}.
Being a semi-analytic subset of the real space, the amoeba
carries a lot of geometric, algebraic, topological, and combinatorial
information on the corresponding algebraic variety~\cite{Mikhalkin}.
The degeneration of the amoeba of a complex algebraic variety
leads to the concept of a tropical algebraic variety providing
an important link between complex analysis and enumerative algebraic
geometry~\cite{Mikhalkin-Topology,Mikhalkin-JAMS}.
Amoebas of algebraic hypersurfaces possess rich analytic and combinatorial
structure reflected in their spines, contours, and tentacles.
They appear in numerous applications in real algebraic geometry, complex analysis,
mirror symmetry and in several other areas. Moreover, they are naturally
linked to the geometry of Newton polytopes, which can be seen in particular with
Viro's patchworking principle (i.e., tropical localization) based on the combinatorics
of subdivisions of convex lattice polytopes. Also, certain tropical varieties
can be seen as a limiting aspect (or "degeneration") of amoebas of algebraic varieties.
For example, complex curves viewed as Riemann surfaces turn to metric graphs
(one-dimensional combinatorial objects), and $n$-dimensional complex varieties
turn to $n$-dimensional polyhedral complexes with some properties (see~\cite{IMS} and~\cite{MS}).
Amoebas have their similar objects in the real torus called {\em coamoebas},
which are the projection of algebraic varieties onto the real torus which
have many interesting and nice properties (see~\cite{NS}).

Despite the simple definition, efficient computation of the amoeba of
a given algebraic variety represents a task of formidable computational
complexity. Various approaches have been recently tried to compute the
shape of an amoeba~\cite{Bogdanov-Kytmanov-Sadykov,Theobald} or approximate it
by simpler geometric objects~\cite{FMMW,Purbhoo,Theobald-deWolff}.
The fundamental problems addressed in numerous papers are the detection
of the topological type of an amoeba, the membership problem for a given
connected component of an amoeba complement, and detection of the
order~\cite{FPT} of such a component in the case of hypersurface amoebas.

Alongside with the definition of unbounded affine amoeba of an algebraic hypersurface,
a competing definition of compactified amoeba has been introduced in~\cite{GKZ}.
While the affine amoeba of a hypersurface is its Reinhardt diagram in the logarithmic
scale, its compactified amoeba is defined to be the image of the hypersurface
under the moment map~\cite{Guillemin} providing a homeomorphism
between the Newton polytope of the defining polynomial of that hypersurface
and the positive orthant of the real vector space. Being topologically equivalent
to the standard affine amoeba, its compactified counterpart often has the
substantial disadvantage of exhibiting complement components of very different
relative size (see examples~\ref{ex:x+30xy+20x^2y+x^3y+y^2}
and~\ref{ex:x+x2+y+xy3+x4y2+3x3y+10xy+10x2y+10xy2+15x2y2+10x3y2} below).
This makes it difficult to work with compactified amoebas in a computationally
reliable way and probably explains the focus of research on affine amoebas.

In the present paper we introduce the definition of an amoeba-shaped polyhedral complex
of an algebraic hypersurface. Like the compactified amoeba, this polyhedral complex
is a subset of the Newton polytope of the defining polynomial of the hypersurface.
Besides, it is a deformation retract of the compactified amoeba and provides the
straightforward solution to the membership problem: the order of a connected component
in its complement is itself a point in this component.
We provide an explicit formula for this polyhedral complex in the case
when the hypersurface~$H$ is optimal~\cite{Bogdanov-Sadykov,FPT}.
We also give a topological description of the complement of the affine amoebas of a class of
algebraic hypersurfaces in combinatorial terms naturally and strongly related
to their Newton polytopes and the coefficients of their defining polynomials.

Pictures of amoebas in the paper have been created in MATLAB~R2017a.
The authors thank D.\,Bogdanov for providing a helpful online tool
for automated generation of MATLAB code which is available for free public use at
http://dvbogdanov.ru/?page=amoeba.

\medskip

{\bf Acknowledgement.} A large part of this paper was written during the second author's visits
to Seoul in 2017. The authors thank Korea Institute for Advanced Study
for providing excellent conditions for research and writing.

%-----------------------------------------------------------------------------------------------------------

\section{Notation and preliminaries}

Throughout the paper, we denote by~$n$ the number of $x\in\C^n$
variables. For $x=(x_1,\ldots,x_n)$ and
$\alpha=(\alpha_1,\ldots,\alpha_n),$ we denote by~$x^\alpha$ the
monomial $x_{1}^{\alpha_1}\ldots x_{n}^{\alpha_n}.$

\begin{definition}
\label{def:amoeba} \rm The {\it amoeba}~$\mathcal{A}_f$ of a
Laurent polynomial~$f(x)$ (or of the algebraic hypersurface $\{
f(x)=0 \}$) is defined to be the image of the
hypersurface~$f^{-1}(0)$ under the map ${\rm Log } :
(x_1,\ldots,x_n)\mapsto (\log |x_1|,\ldots,\log |x_n|).$
\end{definition}

For $n>1,$ the amoeba of a polynomial is a closed unbounded semi-analytic
subset of the real vector space~$\R^n.$
Throughout the paper we will often call such amoebas {\it affine} in order
to distinguish them from the compactified and the weighted compactified
counterparts.

The following result shows that the Newton
polytope~$\mathcal{N}_{p(x)}$ reflects the structure of the
amoeba~$\mathcal{A}_{p(x)}$ \cite[Theorem~2.8 and
Proposition~2.6]{FPT}.

\begin{theorem}\label{3thmfptestimate} {\rm (See \cite{FPT}.)}
Let~$p(x)$ be a Laurent polynomial and let~$\{M\}$ denote the
family of connected components of the amoeba
complement~$^c\!\mathcal{A}_{p(x)}.$ There exists an injective
function $\nu : \{M\}\rightarrow \Z^n \cap \mathcal{N}_{p(x)}$
such that the cone which is dual to~$\mathcal{N}_{p(x)}$ at the
point~$\nu(M)$ coincides with the recession cone of~$M.$ In
particular, the number of connected components of
$^c\!\mathcal{A}_{p(x)}$ cannot be smaller than the number of
vertices of $\mathcal{N}_{p(x)}$ and cannot exceed the number of
integer points in $\mathcal{N}_{p(x)}.$
\end{theorem}

Throughout the paper, the vector $\nu(M)$ will be called the {\it order}
of the connected component~$M$ in the amoeba complement.

\begin{definition}\label{def:compactifiedAmoeba}\rm (See~\cite[Chapter~6]{GKZ}.)
The {\it compactified amoeba}~$\overline{\mathcal{A}}_{f}$ of a Laurent
polynomial $f(x)=\sum\limits_{s\in S} a_s x^s$ (or, equivalently, of the algebraic hypersurface
$\{ f(x)=0 \}$) is defined to be the image of the
hypersurface~$f^{-1}(0)$ under the {\it moment map}
$$
\mu_{S}(x) := \frac{\sum\limits_{s\in S} s\cdot |x^s|}{\sum\limits_{s\in S} |x^s|}.
$$
\end{definition}
By~\cite[Chapter~6]{GKZ}, the compactified amoeba of a polynomial is a closed
subset of its Newton polytope. The amoeba and the compactified amoeba of a polynomial
are homeomorphic. From the computational point of view both have advantages
and shortcomings. It is in general difficult to locate the position of an affine
amoeba in the real space while the integer convex polytope represents a computationally
much more manageable ambient space. On the other hand, some of the connected components
of the complement to a compactified amoeba can be elusively small as illustrated by
Examples~\ref{ex:x+30xy+20x^2y+x^3y+y^2} and~\ref{ex:x+x2+y+xy3+x4y2+3x3y+10xy+10x2y+10xy2+15x2y2+10x3y2}.
Besides, the connected components of the complement to the compactified amoeba
of a polynomial are in general not convex.

\begin{definition}\label{def:optimalAmoeba}\rm (Cf.~\cite[Definition~2.9]{FPT}.)
An algebraic hypersurface $H\subset(\C^{*})^n,$ $n\geq
2,$ is called {\it optimal} if the number of connected components
of its amoeba complement $^c\!\mathcal{A}_{H}$ equals
the number of integer points in the Newton polytope of the
defining polynomial of~$H.$ We will say that a
polynomial (as well as its amoeba) is optimal if its zero locus is
an optimal algebraic hypersurface.
\end{definition}

Since the amoeba of a polynomial does not carry any information on
the multiplicities of its roots, any one-dimensional amoeba (which
is just a finite set of distinct points in $a_1,\ldots, a_k\in\R$)
can be treated as the amoeba of the polynomial
$\prod\limits_{j=1}^{k}(x-e^{a_k})$ all of whose roots are
positive and distinct. Thus Definition~\ref{def:optimalAmoeba} is
trivial in the univariate case. The correct extension of
Definition~\ref{def:optimalAmoeba} to one dimension is to say that
all the roots of the polynomial in question have different
absolute values.

\begin{definition}\label{def:solidAmoeba}\rm
An algebraic hypersurface $H\subset(\C^{*})^n,$ $n\geq
2,$ is called {\it solid} if the number of connected components
of its amoeba complement $^c\!\mathcal{A}_{H}$ equals
the number of vertices of the Newton polytope of the
defining polynomial of~$H.$ We will say that a
polynomial (as well as its amoeba) is solid if its zero locus is
a solid algebraic hypersurface.
\end{definition}

%-----------------------------------------------------------------------------------------------------------

\section{Explicit analytic formula for the amoeba-shaped polyhedral complex}\label{sec:mainResult}

The {\it Newton polytope}~$\mathcal{N}_{p(x)}$ of a Laurent
polynomial~$p(x)$ is defined to be the convex hull in~$\R^n$ of
the support of~$p(x).$ We will often drop some of the subindices
to simplify the notation.
\begin{definition} \rm Following the
ideas of~\cite{Zharkov}, we define the {\it weighted moment map}
associated with the algebraic hypersurface $\{ x\in\C^n : f(x):=
\sum\limits_{s\in S} a_s x^s = 0\}$ through
$$
\mu_{f}(x) := \frac{\sum\limits_{s\in S} s\cdot |a_s|
|x^s|}{\sum\limits_{s\in S} |a_s| |x^s|}.
$$
\label{def:weightedMomentMap}
\end{definition}
It follows from the general theory of moment maps~\cite{Guillemin}
that $\mu_{f}(\C^n)\subseteq \mathcal{N}_{f}.$
\begin{definition}\rm
By the {\it weighted compactified amoeba} of an algebraic
hypersurface $H=\{x\in\C^n : f(x)=0\}$ we will mean the set
$\mu_{f}(H).$ We denote it by $\mathcal{WCA}(f).$
\label{def:weightedCompactifiedAmoeba}
\end{definition}

Recall that the {\it Hadamard power} of order~$r\in\R$ of a
polynomial $f(x)=\sum\limits_{s\in S} a_s x^s$ is defined to be
$f^{[r]}(x):=\sum\limits_{s\in S} a_{s}^{r} x^s.$

\begin{theorem}\label{thm:polyhedralComplex}
Let $f$ be a polynomial in $\mathbb C[x_1^{\pm 1},\cdots, x_n^{\pm 1}]$ with the Newton
polytope~$\mathcal{N}$  such that $|a_{\alpha}|\geq 1$ for every $\alpha\in \Verte(\mathcal{N})$.
Assume that the function which assigns  to each $\alpha\in \mathcal{N}\cap \mathbb{Z}^n$
the real number $\log |a_{\alpha}|$ is concave, and the subdivision of~$\mathcal{N}$ dual to the
tropical hypersurface~$\Gamma$ associated to the  tropical polynomial $f_{trop}$ defined by:
$$
f_{trop}(\zeta) = \max_{\alpha\in\mathcal{N}\cap \mathbb{Z}^n } \{\log|a_{\alpha}| + \langle\alpha,\zeta\rangle\}
$$
is a triangulation. Then the set-theoretical limit
\begin{equation}\label{polyhedralComplexThroughHadamard}
\mathcal{P}_{f}^{\infty} := \lim_{r\rightarrow \infty} WC\mathcal{A}(f^{[r]})
\end{equation}
is a polyhedral  complex. Moreover,  its complement in~$\mathcal{N}$ has the same topology
of the complement of the amoeba~$\mathcal{A}$ of~$f$, i.e.
$\pi_0 (\mathbb{R}^n\setminus\mathcal{A}) = \pi_0 (\mathcal{N}\setminus \mathcal{P}_{f}^{\infty})$.
In particular, if $n=2$ then~$\mathcal{P}_{f}^{\infty}$ is a simplicial complex.
\end{theorem}

\begin{example}\label{ex:x+y+x^2y^2+2xy}
The connected components of the complement of~$\mathcal{P}_{f}^{\infty}$ in
the Newton polytope are not necessarily convex.
\rm The amoeba, the compactified amoeba and the associated
polyhedral complex for the polynomial $x+y+x^2y^2+cxy$ are
depicted in Figures~\ref{fig:x+y+x^2y^2+0,5xy}
and~\ref{fig:x+y+x^2y^2+2xy} for $c=1/2$ and $c=2,$
respectively.

\begin{figure}
\includegraphics[width=4.5cm]{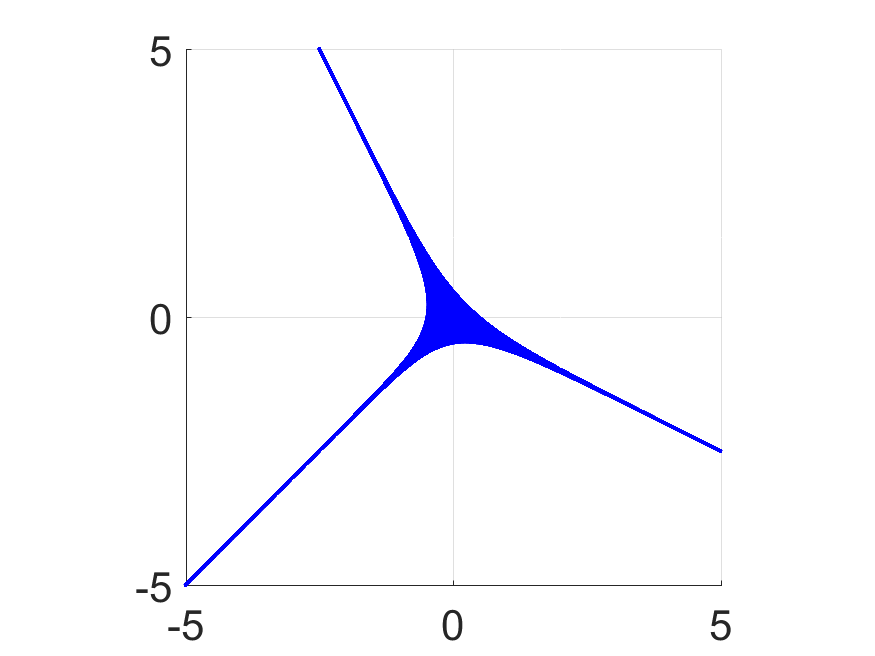}\quad
\includegraphics[width=4.5cm]{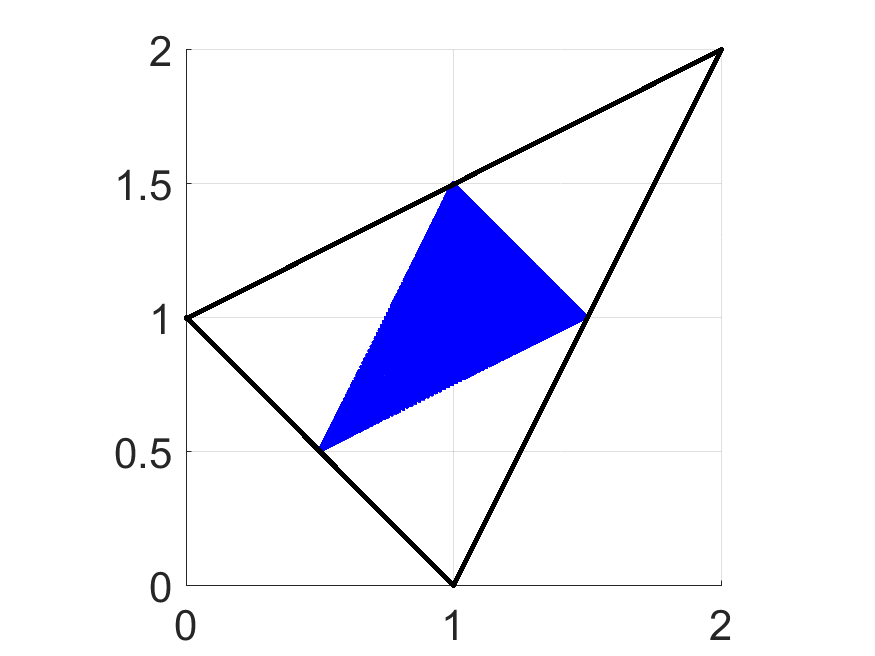}\quad
\caption{The affine and the compactified amoebas of the polynomial
$x+y+x^2y^2+xy/2$. The polyhedral complex coincides with the
compactified amoeba of this polynomial.}
\label{fig:x+y+x^2y^2+0,5xy}
\end{figure}
\begin{figure}
\includegraphics[width=4.5cm]{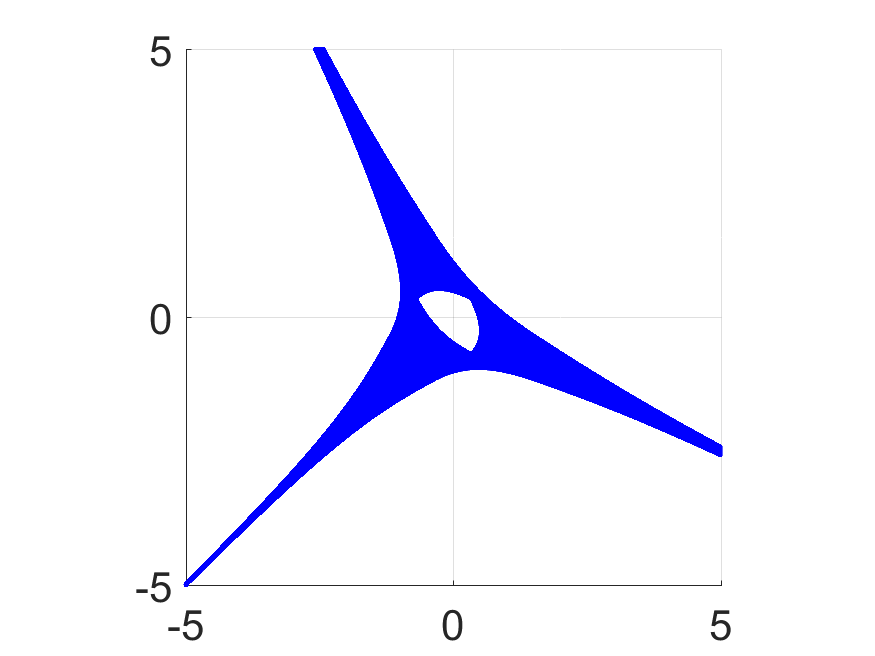}\quad
\includegraphics[width=4.5cm]{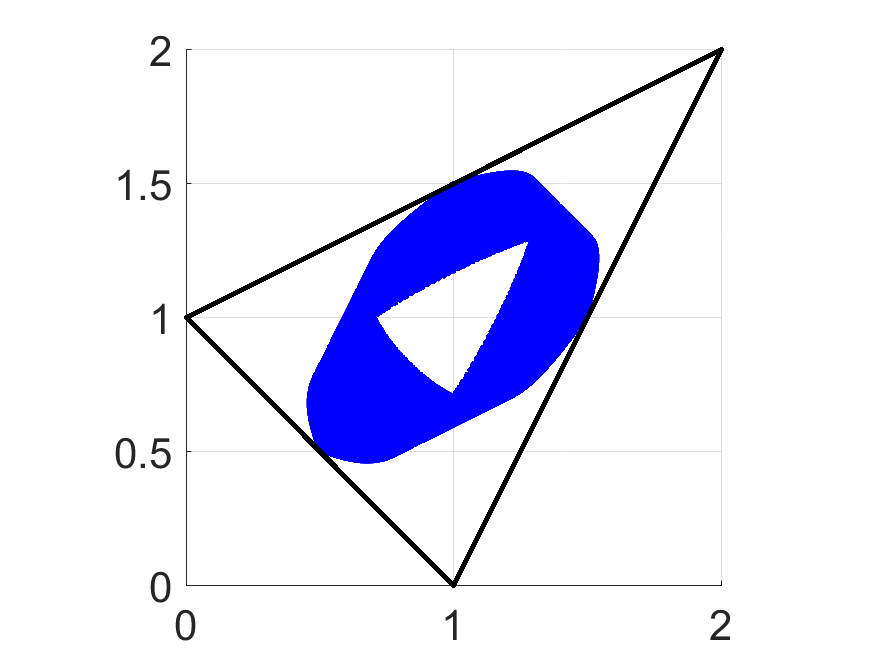}\quad
\includegraphics[width=4.5cm]{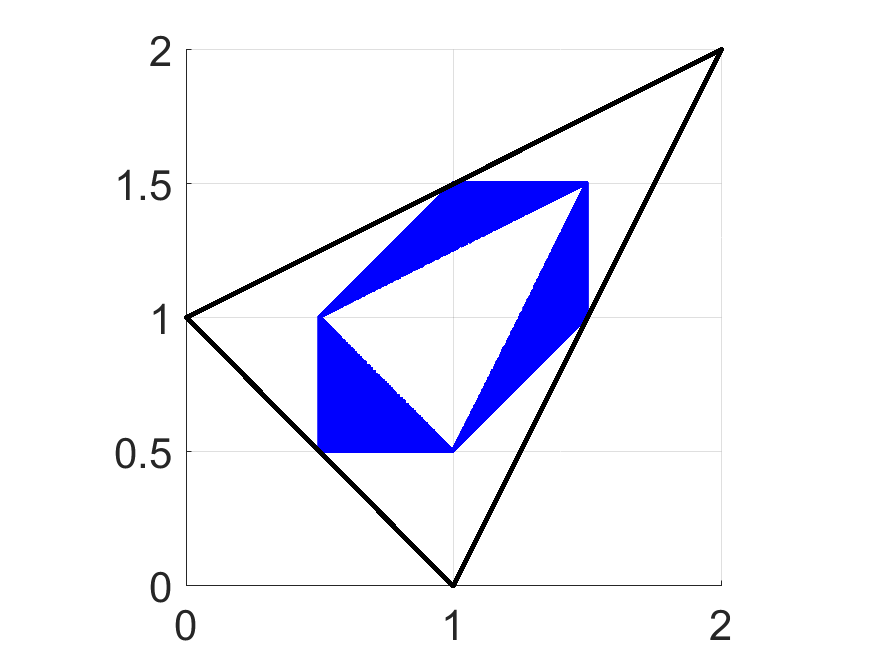}
\caption{The affine amoeba, the compactified amoeba and the
polyhedral complex of the polynomial $x+y+x^2y^2+2xy$}
\label{fig:x+y+x^2y^2+2xy}
\end{figure}

\end{example}

\begin{example}\label{ex:x+y+xy^2+x^2y+cxy}
\rm The amoeba, the compactified amoeba and the associated
polyhedral complex for the polynomial $x+y+xy^2+x^2y+cxy$ are
depicted in Figures~\ref{fig:x+y+xy^2+x^2y+0,5xy}
and~\ref{fig:x+y+xy^2+x^2y+5xy} for $c=1/2$ and $c=5,$ and
respectively.

\begin{figure}
\includegraphics[width=4.5cm]{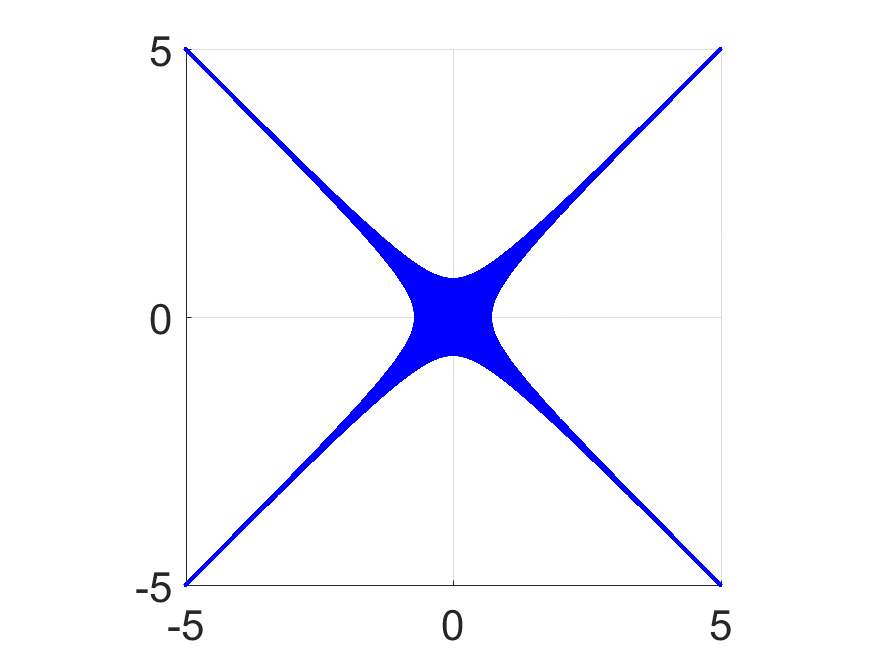}\quad
\includegraphics[width=4.5cm]{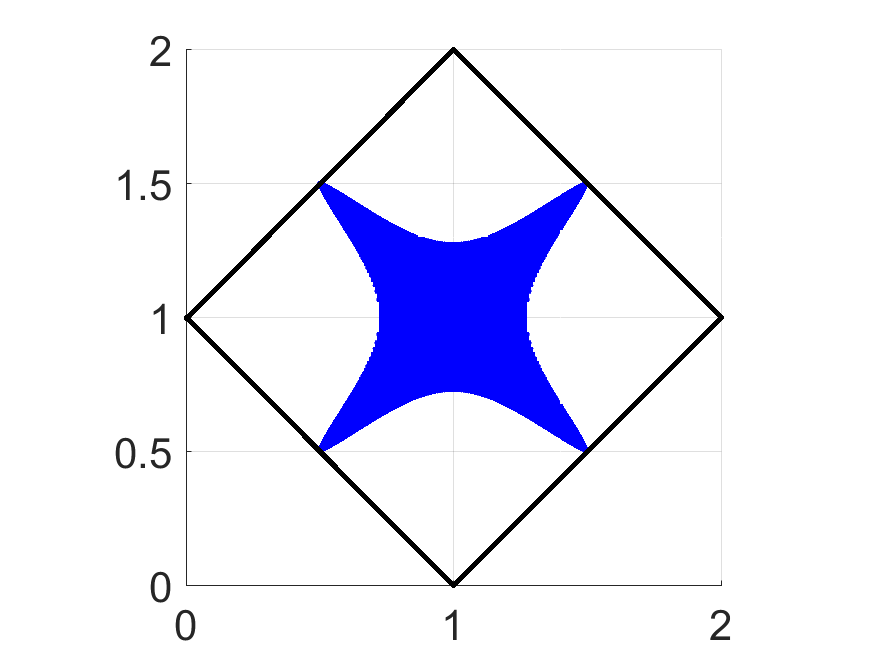}\quad
\includegraphics[width=4.5cm]{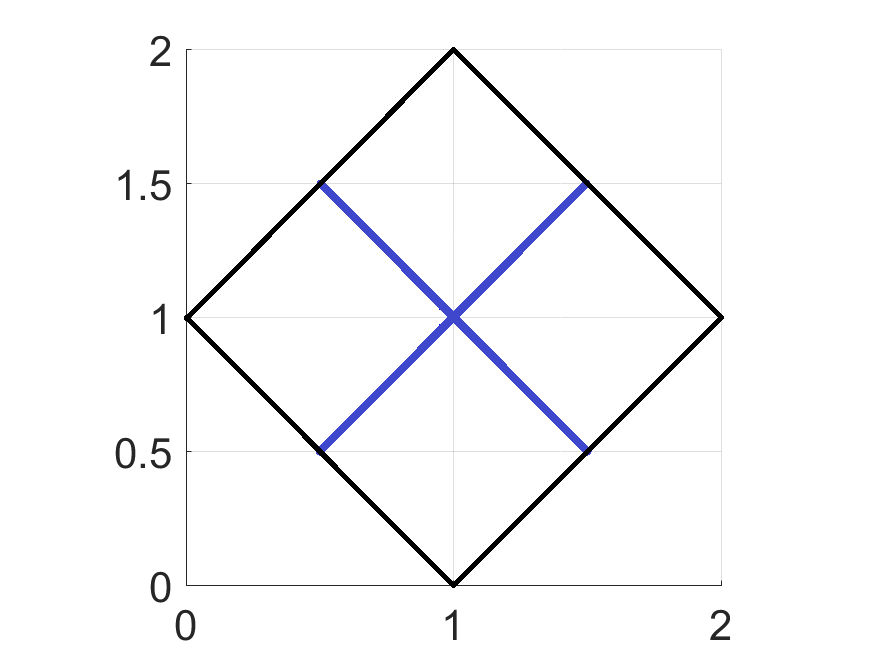}
\caption{The affine amoeba, the compactified amoeba and the
polyhedral complex of the polynomial $x+y+xy^2+x^2y+xy/2$}
\label{fig:x+y+xy^2+x^2y+0,5xy}
\end{figure}
\begin{figure}
\includegraphics[width=4.5cm]{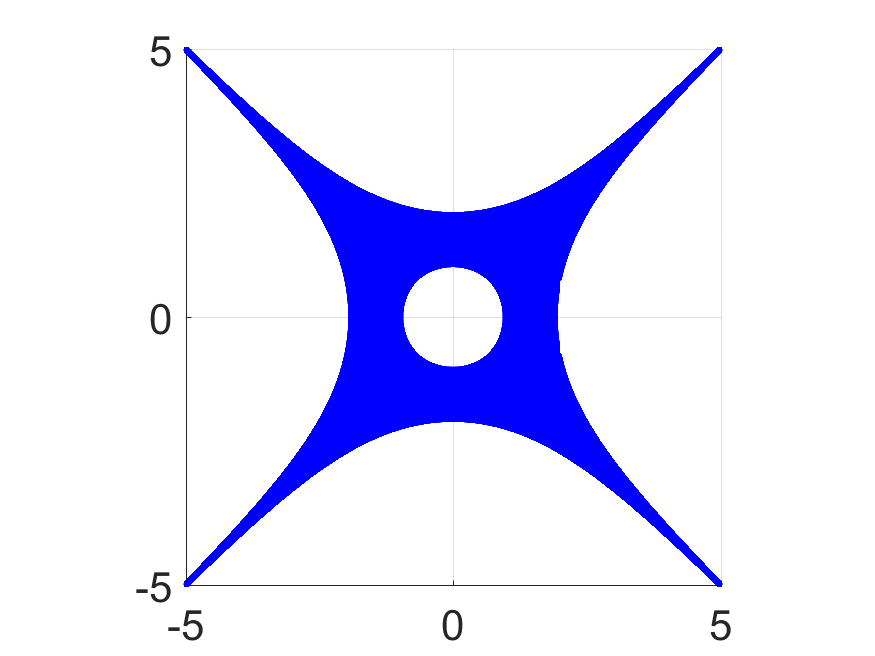}\quad
\includegraphics[width=4.5cm]{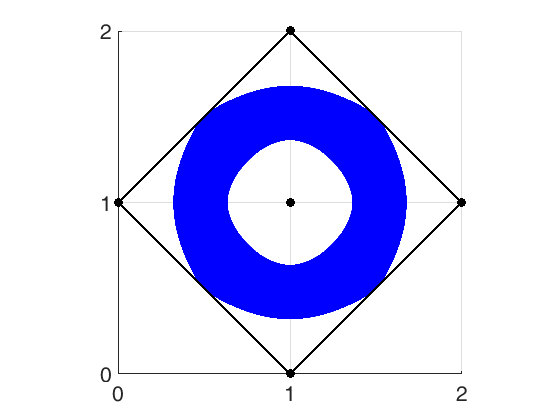}\quad
\includegraphics[width=4.5cm]{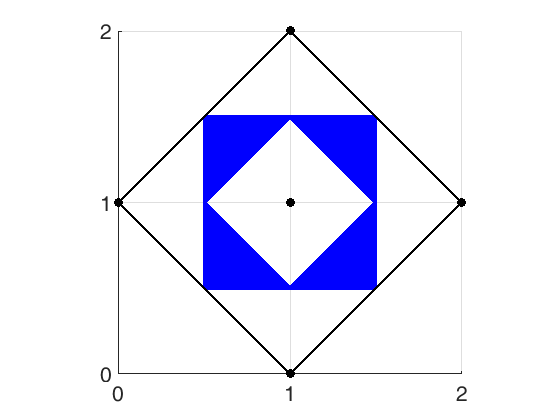}
\caption{The affine amoeba, the compactified amoeba and the
polyhedral complex of the polynomial $x+y+xy^2+x^2y+5xy$}
\label{fig:x+y+xy^2+x^2y+5xy}
\end{figure}

\end{example}

\begin{example}\label{ex:x+30xy+20x^2y+x^3y+y^2}
\rm The zero locus of the polynomial $x+30xy+20x^2y+x^3y+y^2$ is an optimal
hypersurface. Its amoeba, compactified amoeba and the associated
polyhedral complex are depicted in Figure~\ref{fig:x+30xy+20x^2y+x^3y+y^2}.
Although the compactified amoeba is optimal, the three connected components
of its complement that correspond to the vertices of the Newton polytope
are very small in comparison with the two bounded components that fill almost all
of the Newton polygon. In fact, these three components are so small that
it is difficult (yet not impossible) to distinguish them by eye on the presented picture.
This and similar examples motivate the search for a better compact counterpart
of an affine amoeba pursued in the present paper.

\begin{figure}
\includegraphics[width=4.5cm]{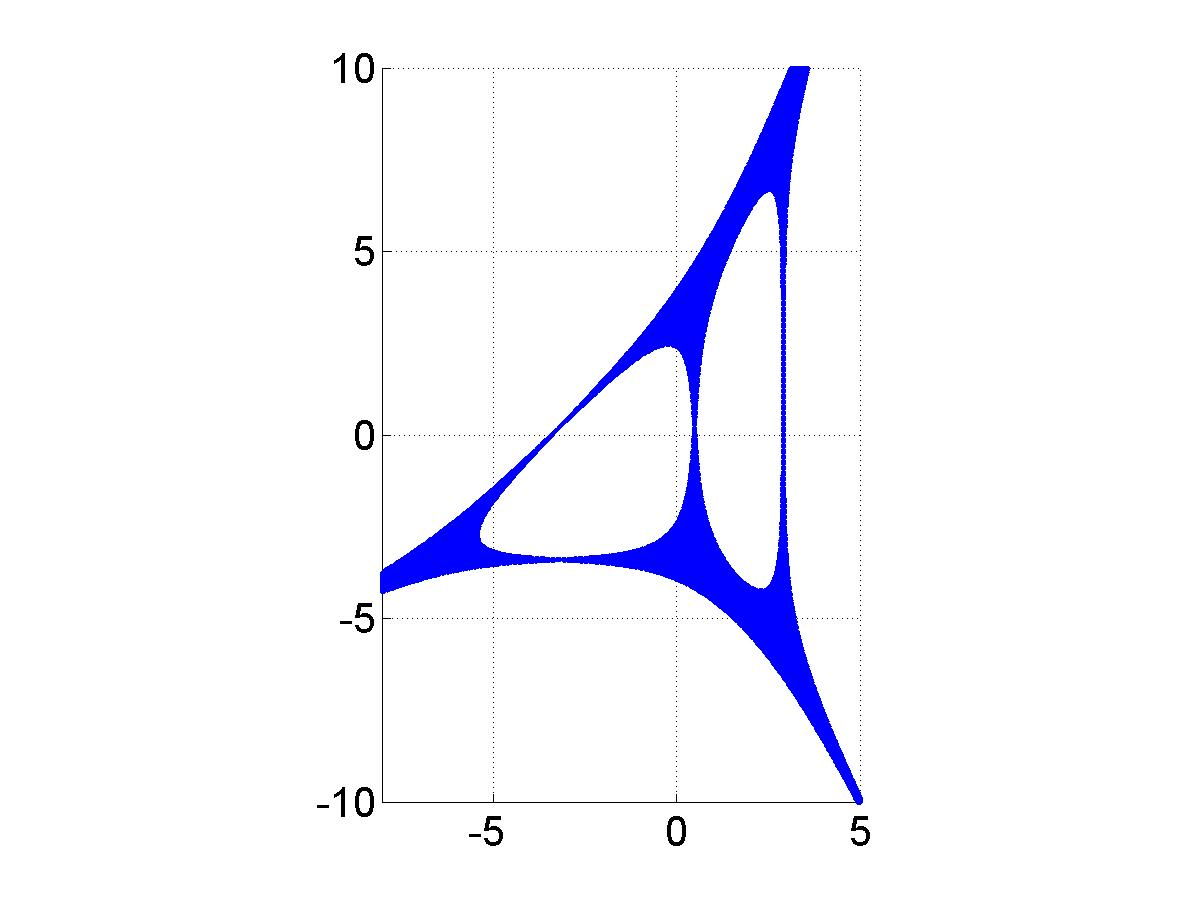}\quad
\includegraphics[width=4.5cm]{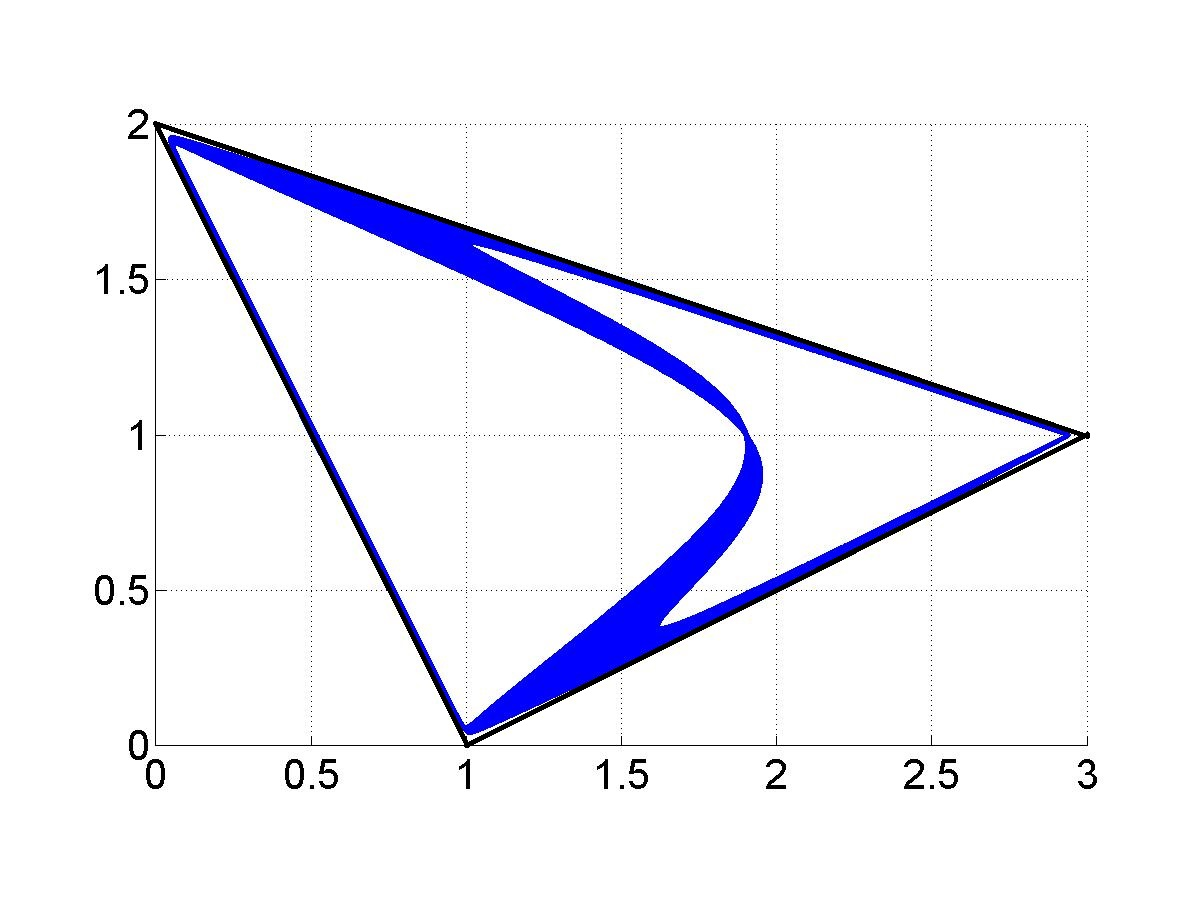}\quad
\includegraphics[width=4.5cm]{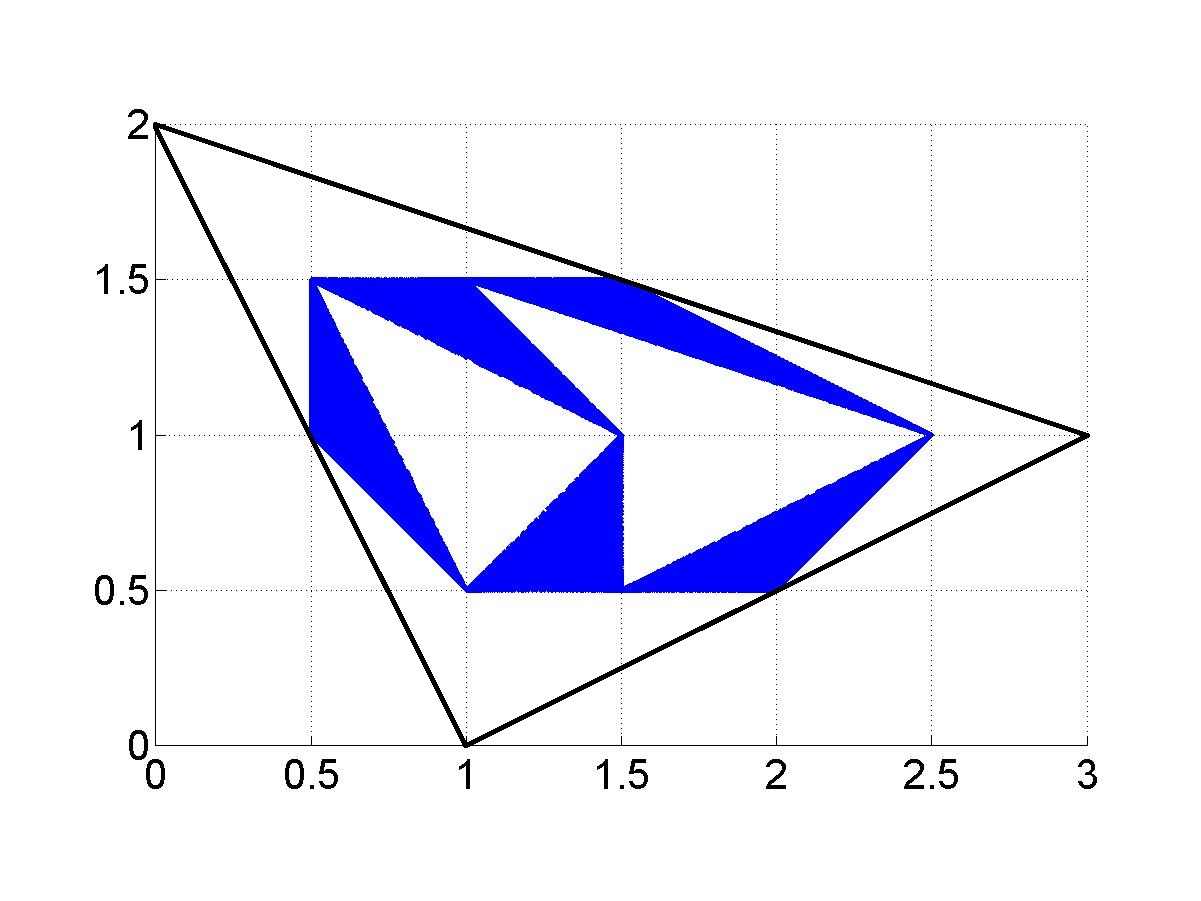}
\caption{The affine amoeba, the compactified amoeba and the
polyhedral complex of the polynomial $x+30xy+20x^2y+x^3y+y^2$}
\label{fig:x+30xy+20x^2y+x^3y+y^2}
\end{figure}

\end{example}

Example~\ref{ex:1+3x+3y+x2y+4x3y+xy2+10x2y2+4xy3} shows that
the sufficient condition of the theorem is not necessary
for the Hadamard power approach to work. Indeed, the polynomial
in this example is neither dense nor optimal.

\begin{proposition}
If $f(x)$ is a vertex polynomial whose Newton polytope~$\mathcal{N}_f$
is a simplex then $\mathcal{P}_{f}^{\infty}$ is given by the convex hull
of the middle points of all 1-dimensional faces of~$\mathcal{N}_f$.
\label{prop:triangles}
\end{proposition}
\begin{proof}
Use a monomial change of variables to reduce to the hyperplane
case and employ an argument parallel to that in the proof
of~\cite[Proposition~4.2]{FPT}.
\end{proof}

\begin{example}\label{ex:hyperplane} A hyperplane. \rm
The polyhedron associated with the hyperplane $\{1+x+y+z=0\}$ is
depicted in Figure~\ref{fig:S(f)hyperplane} inside the unit
simplex. It is combinatorially equivalent to an octahedron.

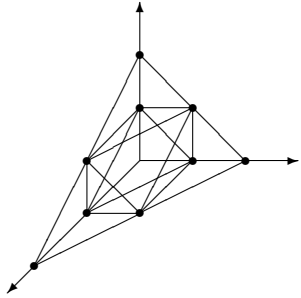
\begin{figure}[htbp]
\begin{center}
\begin{minipage}{4.5cm}
\begin{picture}(120,120)
% The coordinates
   \put(60,60){\vector(-1,-1){50}}
%     \put(17,10){$\scriptscriptstyle x_1$}
   \put(60,60){\vector(1,0){60}}
%     \put(110,52){$\scriptscriptstyle x_2$}
   \put(60,60){\vector(0,1){60}}
%     \put(64,115){$\scriptscriptstyle x_3$}
% The edges of the Newton polytope
   \put(20,20){\line(2,1){80}}
   \put(20,20){\line(1,2){40}}
   \put(60,100){\line(1,-1){40}}
   \put(40,40){\line(2,1){40}}
   \put(40,40){\line(1,2){20}}
   \put(60,80){\line(1,-1){20}}
   \put(80,80){\line(-1,0){20}}
   \put(80,80){\line(0,-1){20}}
   \put(40,60){\line(1,1){20}}
   \put(40,60){\line(0,-1){20}}
   \put(60,40){\line(1,1){20}}
   \put(60,40){\line(-1,0){20}}
   \put(40,60){\line(2,1){40}}
   \put(40,60){\line(1,-1){20}}
   \put(60,40){\line(1,2){20}}
% The vertices of the Newton polytope
   \put(20,20){\circle*{3}}
   \put(60,100){\circle*{3}}
   \put(100,60){\circle*{3}}
   \put(40,40){\circle*{3}}
   \put(60,80){\circle*{3}}
   \put(80,60){\circle*{3}}
   \put(80,80){\circle*{3}}
   \put(60,40){\circle*{3}}
   \put(40,60){\circle*{3}}
\end{picture}
\end{minipage}
\end{center}
\caption{The polyhedral complex~$\mathcal{P}_{f}^{\infty}$ in the
three-dimensional hyperplane case $f(x,y,z)=1+x+y+z$}
\label{fig:S(f)hyperplane}
\end{figure}

\end{example}

%-----------------------------------------------------------------------------------------------------------

\section{Proof of the main theorem}

Let~$t$ be a strictly positive real number and~$H_t$ be the
self-diffeomorphism of~$(\mathbb{C}^*)^n$ defined by
\[
\begin{array}{ccccl}
H_t & : &(\mathbb{C}^*)^n&\longrightarrow&(\mathbb{C}^*)^n,\\
    &   &(x_1,\ldots ,x_n)&\longmapsto&(|x_1|^{\frac{1}{\log t}}
\frac{x_1}{|x_1|}, \ldots ,|x_n|^{\frac{1}{\log t}}
\frac{x_n}{| x_n|} ).
\end{array}
\]
which defines a new complex structure on~$(\mathbb{C}^*)^n$
denoted by $J_t = (dH_t)\circ J\circ (dH_t)^{-1}$ where~$J$ is the
standard complex structure.

\noindent A~$J_t$-holomorphic hypersurface~$V_t$ is a hypersurface
holomorphic with respect to the~$J_t$ complex structure on
$(\mathbb{C}^*)^n$. It is equivalent to say that $V_t = H_t(V)$ where
$V\subset (\mathbb{C}^*)^n$ is an holomorphic hypersurface for the
standard complex structure~$J$ on~$(\mathbb{C}^*)^n$.

Recall that the Hausdorff distance between two closed subsets $A,B$
of a metric space $(E,d)$ is defined by:
$$
d_{\mathcal{H}}(A,B) = \max \{  \sup_{a\in A}d(a,B),\sup_{b\in B}d(A,b)\}.
$$
Here we take $E =\mathbb{R}^n\times (S^1)^n$, with the distance
defined as the product of the
Euclidean metric on~$\mathbb{R}^n$ and the flat metric on~$(S^1)^n$.

\begin{definition} A complex tropical hypersurface $V_{\infty}\subset
(\mathbb{C}^*)^n$ is the limit (with respect to the Hausdorff
metric on compact sets in $(\mathbb{C}^*)^n$) of a  sequence of a
$J_t$-holomorphic hypersurfaces $V_t\subset (\mathbb{C}^*)^n$
when~$t$ tends to~$\infty$.
\end{definition}

The argument map is the map defined as follows:
\[
\begin{array}{ccccl}
\widetilde{\Arg} & : & (\mathbb{C}^*)^n & \longrightarrow&(S^1)^n, \\
                 &   & (x_1,\ldots ,x_n)& \longmapsto&(\widetilde{\arg} (x_1),\ldots ,\widetilde{\arg} (x_n) ).
\end{array}
\]
We use the following notations: if $z = (x_1,x_2,\ldots ,x_n)\in (\mathbb{C}^*)^n$
and $x_j = \rho_j e^{i\gamma_j}$, then
$\widetilde{\Arg} (z) = (\widetilde{\arg} (x_1),\widetilde{\arg} (x_2),\ldots ,\widetilde{\arg} (x_n)) :=
(e^{i\gamma_1},e^{i\gamma_2},\ldots ,e^{i\gamma_n})$ and \\ $\Arg (z) =
(\arg (x_1),\arg (x_2),\ldots ,\arg (x_n)) := (\gamma_1,\gamma_2,\ldots ,\gamma_n)$.

We complexify the valuation map as follows :
\[
\begin{array}{ccccl}
w & : & \K^* & \longrightarrow&\mathbb{C}^*,\\
  &   &   a  &\longmapsto&w(a ) = e^{\val (a) + i \arg (\xi_{-\val (a)})}.
\end{array}
\]
Let  $\widetilde{\Arg}$ be the argument map $\K^*\rightarrow S^1$ defined by: for any $ a \in \K$ with
$\displaystyle{a = \sum_{j\in A_a}\xi_jt^j}$, $\widetilde{\Arg} (a) =
e^{i\arg (\xi_{-\val (a)})}$ (this map extends the map
$\widetilde{\Arg} : \mathbb{C}^*\rightarrow S^1$ defined by $\rho e^{i\theta} \mapsto e^{i\theta}$).

Applying this map coordinate-wise we obtain the map
\[
\begin{array}{ccccl}
W:&(\K^*)^n&\longrightarrow&(\mathbb{C}^*)^n.
\end{array}
\]

Using Kapranov's theorem~\cite{Kapranov} and degeneration of a complex structures,
Mikhalkin gives an algebraic definition of a complex tropical hypersurfaces
(see~\cite{Mikhalkin-JAMS}) as follows:

\begin{theorem}(See~\cite{Mikhalkin-Topology}.) The set $V_{\infty}\subset (\mathbb{C}^*)^n$
is a complex tropical hypersurface if and only if there
exists an algebraic hypersurface $V_{\mathbb{K}}\subset(\K^*)^n$ such that $W(V_{\mathbb{K}}) = V_{\infty}$.
\end{theorem}

Let $\Log_{\mathbb{K}}(x_1,\ldots ,x_n)=(\val (x_1),\ldots ,
\val (x_n))$, which means that~$\K$ is equipped with the
norm defined by $\Vert z\Vert_{\mathbb{K}}= e^{\val (z)}$ for any
$z\in\K^*$.
Then we have the following commutative diagram:
\begin{equation}
\xymatrix{
(\K^*)^n\ar[rr]^{W}\ar[dr]_{\Log_{\mathbb{K}}}&&(\mathbb{C}^*)^n\ar[dl]^{\Log}\cr
&\mathbb{R}^n
}\nonumber
\end{equation}

\begin{theorem}(See~\cite{Mikhalkin-Topology,Mikhalkin-JAMS}.)
Let~$f$ be a polynomial in $\K[x_1,\ldots,x_n]$. Then we have the following:
$$
\lim_{t\rightarrow\infty} H_t(V_{f_{\frac{1}{t}}}) = W(V_f)
$$
with respect to the Hausdorff metric on compact sets in~$(\mathbb{C}^*)^n$,
and where~$f_{\frac{1}{t}}$ is the polynomial in $\mathbb{C}[x_1,\ldots,x_n]$
where we fixed  the variable $t>>1$ of the coefficients of~$f$.
\end{theorem}

It was shown by Mikhalkin that if~$f$ and~$f'$ are polynomials in $\K[x_1,\ldots,x_n]$
such that the coefficients of~$f'$ are the leading monomials of the coefficients of~$f$
then  $W(V_f) = W(V_{f'})$. Also,  if~$\zeta$ is a point in $\Log_{\mathbb{K}}(V_f)$ and~$U_\zeta$ is an open
neighborhood of~$\zeta$ such that  $U_\zeta\cap \Val (V_f)$ is a cone centered  at~$\zeta$,
then $W(V_f)\cap \Log^{-1}(U_\zeta)$ is equal to $W(V_{f^{\mathcal{N}_\zeta}})\cap\Log^{-1}(U_\zeta)$
where~$\mathcal{N}_\zeta$ is a cell of the subdivision that is dual to $\Log_{\mathbb{K}}(V_f)$ and contains~$\zeta$.
By Viro's Theorem~\ref{Viro} there exists $U\in \mathbb{R}^n$ such that for
$t$ sufficiently large $\Log^{-1}(U_\zeta)\cap H_t(V_{f_{\frac{1}{t}}})$ is isotopic to
$\Log^{-1}(U_\zeta)\cap e^UH_t(V_{f^{\mathcal{N}_\zeta}})$ and we have
the following diagram:

\begin{equation}
\xymatrix{
\Log^{-1}(U_\zeta)\cap H_t(V_{f_{\frac{1}{t}}})\ar[rr]^{\simeq}\ar[d]_{t\rightarrow \infty}&&
\Log^{-1}(U_\zeta)\cap e^UH_t(V_{f^{\mathcal{N}_\zeta}})\ar[d]^{t\rightarrow\infty}\cr
\Log^{-1}(U_\zeta)\cap W(V_f)\ar[rr]^{=}&&\Log^{-1}(U_\zeta)\cap e^UW(V_{f^{\mathcal{N}_\zeta}}).
}     \nonumber
\end{equation}

\begin{theorem}(See~\cite{V-90}.)\label{Viro}
Let $f = \sum a_jz^j$ be a polynomial in $\K[x_1,\ldots,x_n]$ where $a_j = \sum\xi_r^jt^r$.
Let~$\zeta$ be a point in $\Log_{\mathbb{K}} (V_f)$ and~$U_\zeta$ be an open  neighborhood of~$\zeta$ such that
$U_\zeta\cap \Log_{\mathbb{K}} (V_f)$ is a cone centered  at~$\zeta$. Let~$\mathcal{N}_\zeta$ be a cell
of the subdivision that is dual to $\Log_{\mathbb{K}}(V_f)$ and contains~$\zeta$, and let $f^{\mathcal{N}_\zeta}(z)$
be the polynomial in $\mathbb{C}[x_1,\ldots,x_n]$ defined as follows:
$$
f^{\mathcal{N}_\zeta}(z) = \sum_{j\in \mathcal{N}_\zeta}\xi_{-val(a_j)}^jz^j.
$$
Then there exists $U\in\mathbb{R}^n$ such that for any sufficiently
large~$t$,  $\Log^{-1}(U_\zeta)\cap H_t(V_{f_{\frac{1}{t}}})$ is isotopic to
$\Log^{-1}(U_\zeta)\cap e^UH_t(V_{f^{\mathcal{N}_\zeta}})$.
\end{theorem}

Let now~$f$ be a polynomial in $\mathbb C[x_1^{\pm 1},\cdots , x_n^{\pm 1}]$
with the Newton polytope~$\mathcal{N}$ and amoeba~$\mathcal{A}$ with spine~$\Gamma_f$.
Moreover, assume that the subdivision~$\tau_f$ of~$\mathcal{N}$ dual to~$\Gamma_f$ is a triangulation.

The moment map~$\mu_{f}$ is an embedding with image the interior of~$\mathcal{N}$,
and we have the following commutative diagram:
\begin{equation}
\xymatrix{
(\mathbb{C}^*)^n\ar[rr]^{\Log}\ar[dr]_{\mu_{f}}&&\mathbb{R}^n\ar[dl]^{\Psi_f}\cr
&\mathcal{N}
}\nonumber
\end{equation}
The maps~$\Log$ and~$\mu_{f}$ both have orbits~$(S^1)^n$ as fibers,
and we obtain a reparametrization of~$\mathbb{R}^n$  which we denote
by~$\Psi_{f}$ (see~\cite{GKZ}).

We can now finish the proof of main Theorem~\ref{thm:polyhedralComplex}.
The fact that the function $\alpha \longmapsto \log|a_{\alpha}|$
is concave implies that the same hypothesis holds for any
Hadamard product~$f^{[r]}$ and any positive number~$r$.
Moreover, the  subdivisions of~$\mathcal{N}$ dual to the tropical hypersurfaces
$\Gamma^{[r]}$ associated to~$f^{[r]}$ are all the same (because the $\Gamma^{[r]}$'s
have all the same combinatorial type). More precisely, this means that the complement
in~$\mathbb{R}$ of~$\mathcal{A}$ and~$\mathcal{A}_r$  have the same number of connected components.
In particular, it means that if the amoeba~$\mathcal{A}$ of~$f$ is optimal
then all amoebas~$\mathcal{A}_r$ of~$f^{[r]}$ are optimal.
Moreover, we know that  any lattice simplex can
be identified with the standard simplex by an element of the group $AGL_n(\mathbb{Z})$
of affine linear transformations of~$\mathbb{R}^n$ whose linear part belongs to $GL_n(\mathbb{Z})$.
This means that when we pass to the limit as~$r$ tends to infinity of~(\ref{polyhedralComplexThroughHadamard}),
and we take the truncation of the polynomial to an element of the subdivision~$\tau$
(which is a simplex by our hypothesis) dual to~$\Gamma$ we obtain,  up to a linear transformation,
the polyhedron corresponding to the standard hyperplane.  For simplicity, let~$\zeta$
be a vertex of~$\Gamma^{[r]}$ for sufficiently large~$r$, and~$U_\zeta$ be an open neighborhood of~$\zeta$.
Then by Viro's Theorem \ref{Viro} we have $\Log^{-1}(U_\zeta)\cap V_{f^{[r]}}$ is isotopic to
$\Log^{-1}(U_\zeta)\cap e^UV_{f^{[r];\mathcal{N}_\zeta}}$ for some $U\in \mathbb{R}^n$ where $f^{[r];\mathcal{N}_\zeta}$
is the truncation of~$f^{[r]}$ to the dual~$\mathcal{N}_\zeta$ of the vertex~$\zeta$, which is a simplex by hypothesis.
But this last set is, up to an affine transformation, the polyhedral corresponding to a hyperplane.
Using Viro's patchwork we obtain a polyhedral complex.
In particular, if $n=2$, we obtain a simplicial complex, because the polyhedron corresponding
to the line is a $2$-dimensional simplex  (i.e. a triangle in this case).~\hfill~$\square$

%-----------------------------------------------------------------------------------------------------------

\section{The polyhedral complex of a general algebraic hypersurface}\label{sec:topologicalRelation}

The explicit analytic formula~(\ref{polyhedralComplexThroughHadamard}) for the
polyhedral complex of an algebraic hypersurface cannot work in the general case.
Indeed, if (the absolute values of) all coefficients of its defining polynomial~$f$ are
equal to~1, then the Hadamard power in~(\ref{polyhedralComplexThroughHadamard})
yields nothing new and~$\mathcal{P}_{f}^\infty$ will in general remain curved.
Example~\ref{ex:1+3x+3y+x2y+4x3y+xy2+10x2y2+4xy3} below show however that the condition
in Theorem~\ref{thm:polyhedralComplex} is far from being necessary. Numerous computer
experiments suggest the following conjecture.

\begin{conjecture}\label{conj:thePerfectPolyhedralComplex}
Let $f(x_1,\ldots,x_n)\in\C[x_{1}^{\pm 1},\ldots,x_{1}^{\pm 1}]$ be a Laurent polynomial.
Denote by $\overline{\mathcal{A}}_{f}\subset\mathcal{N}_f$ its compactified amoeba and by~$\{M\}$
the set of (nonempty) connected components of the complement of~$\overline{\mathcal{A}}_{f}$ in
the Newton polytope~$\mathcal{N}_f.$ We furthermore denote by~$\nu(M)\in\mathcal{N}_f\cap\Z^n$
the order~\cite{FPT} of such a component.

There exists a polyhedral complex~$\mathcal{P}_f\subset\mathcal{N}_f$
with the following properties:

1. The polyhedral complex~$\mathcal{P}_f$ is a deformation retract
of the compactified amoeba~$\overline{\mathcal{A}}_{f}.$

2. For any complement component~$M$ of~$\mathcal{N}_f\setminus\overline{\mathcal{A}}_{f}$
the only integer point that belongs to this component is its order: $M\cap\Z^n = \nu(M).$

\begin{example}\label{ex:1+3x+3y+x2y+4x3y+xy2+10x2y2+4xy3}
\rm
Consider the bivariate polynomial $f(x,y)=1+3x+3y+x^2y+4x^3y+xy^2+10x^2y^2+4xy^3.$
The amoeba~$\mathcal{A}_f$, the compactified amoeba~$\overline{\mathcal{A}}_f$,
and the polyhedral complex~$\mathcal{P}_{f}^{\infty}$ are shown in Figure~\ref{fig:1+3x+3y+x2y+4x3y+xy2+10x2y2+4xy3}.
We observe that this polynomial is not dense as the integer point $(1,1)$
does to belong to its support. Moreover, this polynomial is not optimal since
the complement of its amoeba lacks bounded components of orders $(1,1),(1,2)$ and $(2,1).$
The set of integer points in the Newton polygon~$\mathcal{N}_f$ that do not belong
to the polyhedral complex~$\mathcal{P}_{f}^{\infty}$ consists of the orders of connected components
in $\mathcal{N}_f\setminus\overline{\mathcal{A}}_f.$ In this example, the polyhedral
complex~$\mathcal{P}_{f}^{\infty}$ has been computed as the limit of the weighted compactified
amoebas of the Hadamard powers of~$f(x,y).$

\begin{figure}
\includegraphics[width=4.5cm]{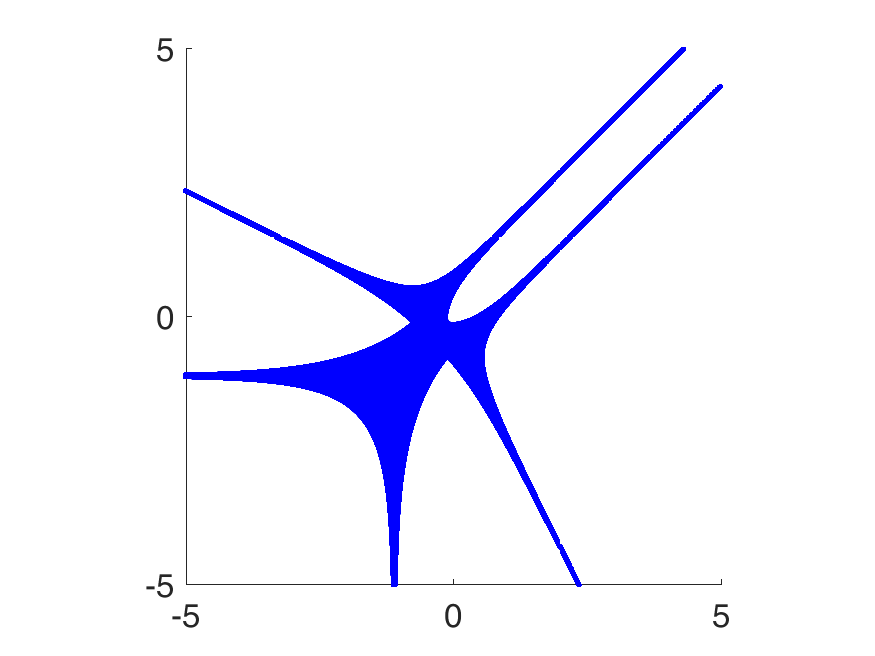}\quad
\includegraphics[width=4.5cm]{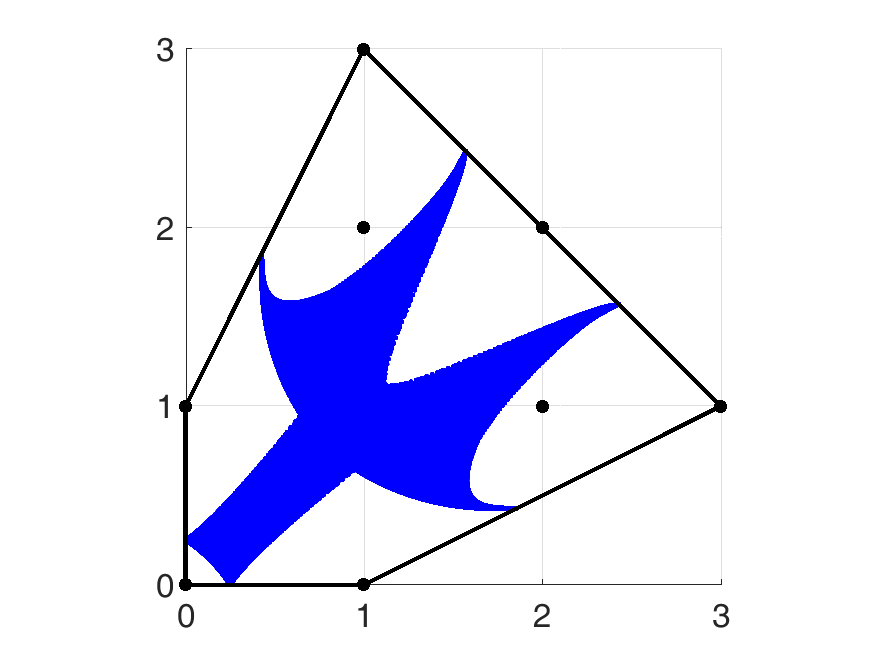}\quad
\includegraphics[width=4.5cm]{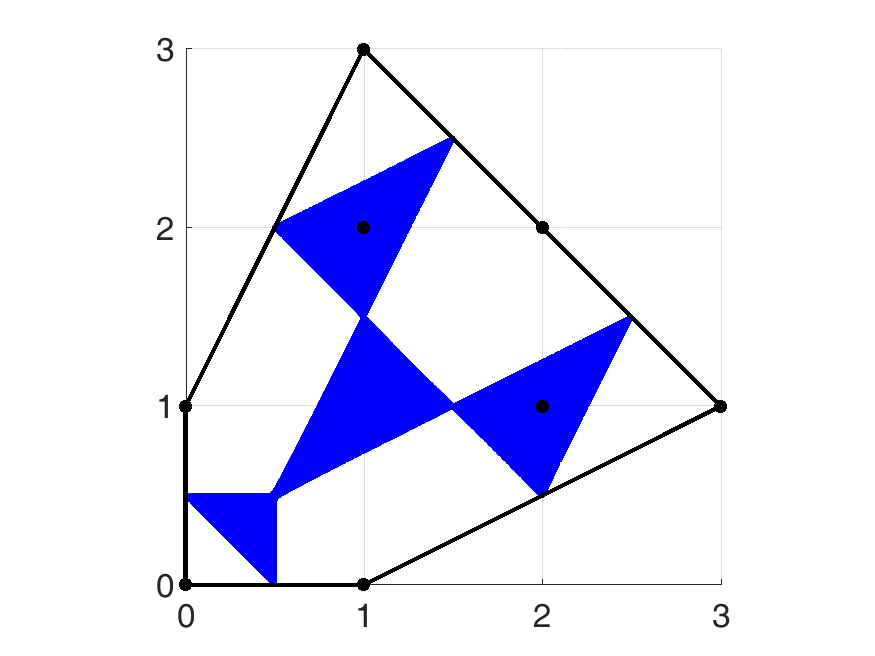}
\caption{The affine amoeba, the compactified amoeba and the
polyhedral complex of the polynomial $1+3x+3y+x^2y+4x^3y+xy^2+10x^2y^2+4xy^3$}
\label{fig:1+3x+3y+x2y+4x3y+xy2+10x2y2+4xy3}
\end{figure}

\end{example}

\end{conjecture}

Another deformation retract of an amoeba, its spine, has been
defined and studied in~\cite{Passare-Rullgaard}. The spine of
the (affine) amoeba of a polynomial can be defined as the set where
certain piecewise linear approximation of the Ronkin
function~\cite{Passare-Rullgaard} associated with this polynomial is nonsmooth.
The spine of an amoeba is a proper subset of the amoeba itself
and thus inherits all the problems that arise when one needs to
determine the topological type of an amoeba or its position in
the ambient affine space.

Polyhedral complexes related to amoebas have also been used in~\cite{Mikhalkin-Topology}
for the analysis of topology of nonsingular algebraic hypersurfaces in projective spaces.
The approach developed in~\cite{Mikhalkin-Topology} is based on Viro's patchworking
and allows one to treat a complex algebraic hypersurface as a singular fibration over
a polyhedral complex, the generic fiber being isomorphic to a smooth torus. This polyhedral
complex is a subset of the Newton polytope of the defining polynomial of the algebraic
hypersurface and is dual to a certain lattice subdivision of this polytope.
However, it is general different from~(\ref{polyhedralComplexThroughHadamard})
and has a different dimension. Besides, it is defined through patchworking rather
than the explicit formula~(\ref{polyhedralComplexThroughHadamard}).

\begin{example}\label{ex:x+x2+y+xy3+x4y2+3x3y+10xy+10x2y+10xy2+15x2y2+10x3y2} \rm
The zero locus of the polynomial $p(x,y)=x+x^2+y+xy^3+x^4y^2+3x^3y+10xy+10x^2y+10xy^2+15x^2y^2+10x^3y^2$
is an optimal hypersurface. The computation of its affine amoeba requires
considerable accuracy due to the very different relative size of
the bounded connected components of its complement
(see Figure~\ref{fig:x+x2+y+xy3+x4y2+3x3y+10xy+10x2y+10xy2+15x2y2+10x3y2}).
The compactified amoeba of this polynomial is also a difficult set to compute
since the vertex components are mapped by the moment map to very small regions
inside the Newton polygon. Figure~\ref{fig:x+x2+y+xy3+x4y2+3x3y+10xy+10x2y+10xy2+15x2y2+10x3y2}
features the evolution of the weighted compactified amoeba of the Hadamard powers of
the polynomial $p(x,y)$ as the exponent ranges from~1 to~3.
The rightmost down picture in Figure~\ref{fig:x+x2+y+xy3+x4y2+3x3y+10xy+10x2y+10xy2+15x2y2+10x3y2}
shows the weighted compactified amoeba of the 3rd Hadamard
power of a deformation of $p(x,y).$ The small bounded component of the
complement to this set shrinks and vanishes precisely at its order, that is,
at the point $(2,2)\in\mathcal{N}_p.$

\begin{figure}
\includegraphics[width=4.5cm]{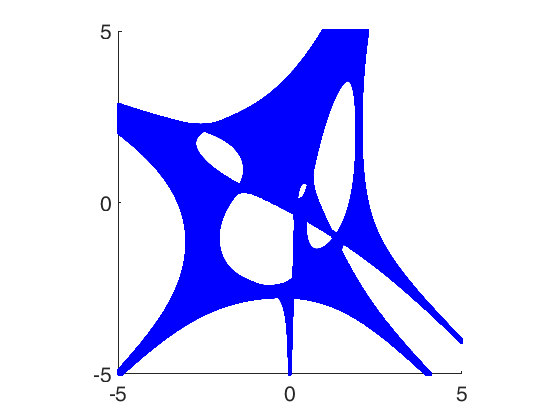} \quad
\includegraphics[width=4.5cm]{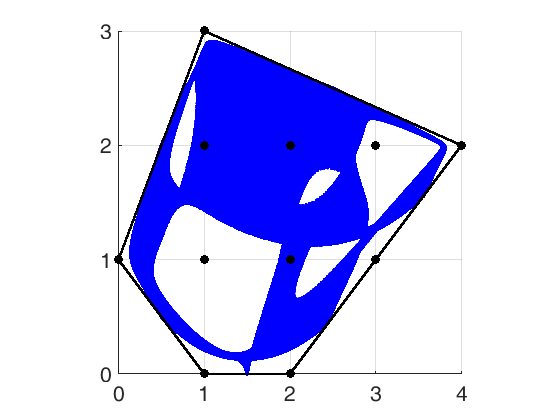} \quad
\includegraphics[width=4.5cm]{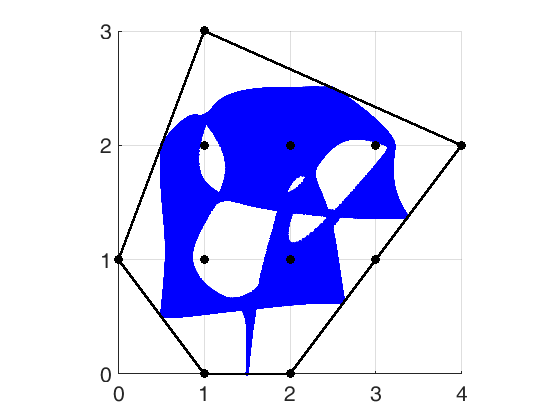} \quad
\includegraphics[width=4.5cm]{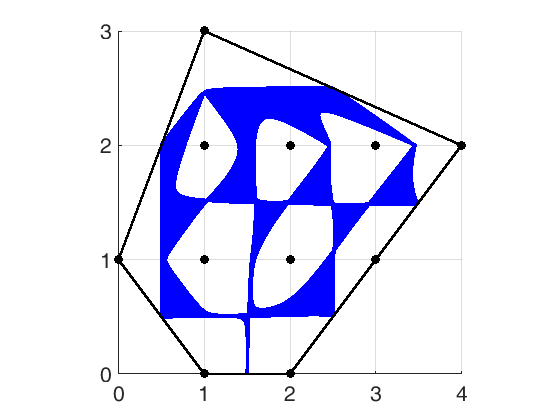} \quad
\includegraphics[width=4.5cm]{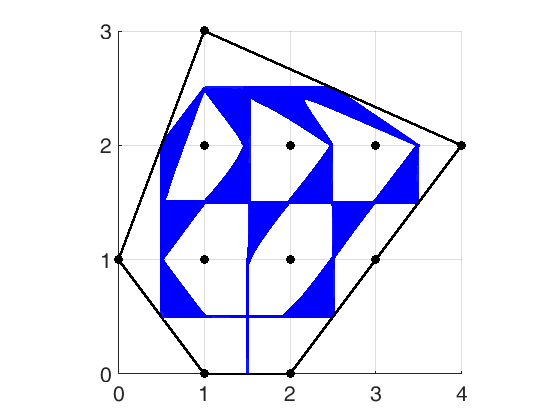} \quad
\includegraphics[width=4.5cm]{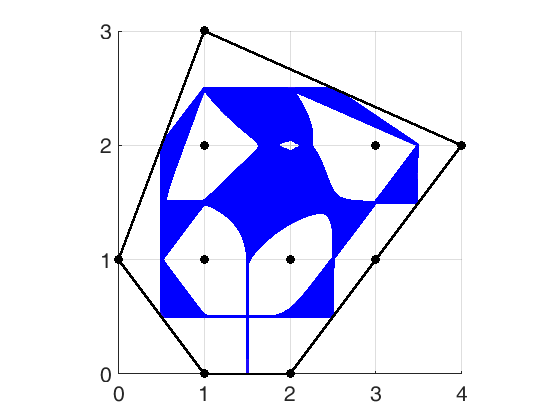}
\caption{The affine amoeba, the compactified amoeba, the
weighted compactified amoebas of the 1st, 2nd and 3rd Hadamard powers of the polynomial
$x+x^2+y+xy^3+x^4y^2+3x^3y+10xy+10x^2y+10xy^2+15x^2y^2+10x^3y^2,$ and a bounded component
of its deformation vanishing at the lattice point $(2,2)$}
\label{fig:x+x2+y+xy3+x4y2+3x3y+10xy+10x2y+10xy2+15x2y2+10x3y2}
\end{figure}

\end{example}

%-----------------------------------------------------------------------------------------------------------

\end{document}